\documentclass[12pt]{amsart}
\usepackage{a4wide}
\usepackage{amsmath,amssymb}

\newtheorem{thm}{Theorem}[section]
 \newtheorem{cor}[thm]{Corollary}
 \newtheorem{lem}[thm]{Lemma}
 \newtheorem{prop}[thm]{Proposition}
 \theoremstyle{definition}
\newtheorem{defn}[thm]{Definition}

\newtheorem{ex}[thm]{Example}

\newtheorem{qu}[thm]{Question}
 \theoremstyle{remark}
 
 \numberwithin{equation}{section}

\newcommand{\la}[1]{\mbox{\rm  lann}_R({#1})}
\newcommand{\La}[2]{\mbox{\rm  lann}_{#1}({#2})}
\newcommand{\lS}[1]{\mbox{\rm  Sing}_l({#1})}
\begin{document}
  \date{\today}
  \author[   Ko\c{s}an,  Matczuk]{    M. Tamer  Ko\c{s}an,   Jerzy Matczuk
\address{Department of Mathematics, Gazi University, Ankara,  Turkey}
\email{  tkosan@gmail.com }
\address{Department of Mathematics, University of Warsaw,  ul. Banacha  2, 02-097 Warsaw, Poland}
\email{jmatczuk@mimuw.edu.pl}}
\title{On  fusible rings}
\keywords{Fusible ring, the classical quotient ring, Goldie ring }  \subjclass[2000]{16N20, 16U60}

 \thanks{ }
  \maketitle

\begin{abstract}  We answer in negative two of questions posed in \cite{GM}.  We also establish a new characterization of semiprime left Goldie rings by  showing that a semiprime ring $R$  is left Goldie iff it is regular left fusible and has finite left Goldie dimension.

 \end{abstract}

\section{\bf Introduction}

Throughout this paper, $R$ denotes an associative ring with unity. For a nonempty subset $S\subseteq R$, $\la{A}$  stands for
the left annihilator   of $A$ in $R$, i.e. $\la{A}=\{r\in R\mid aA=0\}$.

An element $a \in R$ is   a left zero-divisor if  $\la{a} \ne 0$. Elements which are not left zero-divisor are  called left regular.

Ghashghaei and McGovern in \cite{GM} introduced and investigated fusible rings. A nonzero   element $r$ of a ring $R$ is called left fusible if $r$ can be presented in a form $r=c+w$, where $c$ is a left zero-divisor and $w$ is left regular. In  such situation   we will say that $r=c+w$ is a left fusible decomposition of $r$. Our ring $R$ is called left fusible if every nonzero  element of $R$ is left fusible.   Let us point out that our notation differs from the one introduced in  \cite{GM} "left zero divisors" are right zero divisors in the meaning of \cite{GM}. Thus our left fusible rings are right fusible in the language of \cite{GM}.  Clearly both domains and clean rings   are examples of fusible rings (i.e., left and right fusible). It was proved in \cite{GM} that polynomial rings and matrix rings over left fusible rings are left fusible.

 The aim of the paper is to continue a study of fusible and related rings. In particular, we answer in negative questions (1) and (3) posed in \cite{GM}. We also introduce regular left fusible rings, a class which is slightly wider than that of left fusible rings and discuss   relations  between those two classes. As a result a new description of semiprime left Goldie rings is given in Theorem \ref{Goldie}. We also prove (cf. Theorem \ref{rlf matrices}) that the regular left fusible ring property lifts to matrix rings for rings having left quotient rings.    Some questions are formulated.

\newpage
 \section{\bf Results}
We begin with the following definition.
\begin{defn}
  A ring $R$ is said to be regular left fusible if for any nonzero element $r\in R$ there exists a regular (i.e., left and right regular) element $s\in R$  such that the element $sr$ is left fusible, i.e. $sr=c+w$, where $c$ is a left zero divisor and $w$ is left regular.
  \end{defn}
  Since every left regular element is left fusible, our definition reduces to the requirement that for every nonzero left zero divisor $a$ there exists a regular element $s$ in $R$ such that the element $sa$ is left fusible.
 It is also clear that every left fusible ring is regular left fusible (it is enough to take $s=1$). The following example shows that the class of left fusible rings is strictly smaller than the one of  regular left fusible rings.
 \begin{ex}\label{example Cohn}
  Let $F$ be a field and set $R=F<x,y\mid x^2=0>$.  By \cite[Theorem 1]{C}, the set of all left zero divisors of $R$ is equal  to $xR$. In particular,  a difference of any  two left zero divisors is a left zero divisor as well. Thus $R$ is not left fusible.    Notice that $y$ is a regular element of $R$ and $yr=0+yr$ is a left fusible decomposition of $yr$, for any nonzero element $r\in R$. Thus $R$ is  a regular left fusible ring. Similarly, one can see that $R$ is not right fusible but it is  regular right fusible.
 \end{ex}

We will use the above example  to answer  Question (1) of \cite{GM}. Before doing so let us recall some notions.  A ring $R$ is a left p.p. ring if any principal left ideal of $R$ is projective and
 $R$ is said to be   left p.q.-Baer (i.e.,  left principally quasi-Baer)  if the left annihilators of
principal left ideals  are generated as left ideals  by an idempotent. Birkenmeier et al. in \cite[Corollary 1.15]{Bir} proved that a ring
$R$ is an abelian (i.e., with all idempotents   central) left p.p. ring if and only if $R$ is a reduced p.q.-Baer ring.   It was proved in \cite[Corollary 2.15]{GM} that  every abelian  left (right) p.p. ring   is fusible and that, in general,   p.q.-Baer  rings do not have to be left fusible (see \cite[Example 2.17]{GM}). Question (1) of \cite{GM} asks whether  every abelian p.q.-Baer ring is fusible. The following example shows that the answer to this question is negative.

\begin{ex}\label{answ q1}
 Let $R=F<x,y\mid x^2=0>$ be the ring defined in Example \ref{example Cohn}.  Since elements from $yR$ are left regular, we get $ayb\ne 0$, for any nonzero  $a,b\in R$. Thus  $R$ is a prime ring, so  it is also a p.q.-Baer ring.  The only idempotents of $R $ are 0 and 1 (so it is abelian). This can be checked directly or one can apply a classical result of Herstein (cf. \cite{H}).   By \cite[Example 6.]{CPL}, the set of all nilpotent elements of $R$ is equal to $xRx$. The subring $F[x]$ is invariant under all inner  automorphisms adjoint to elements
of the form $(1+n)$, $n\in xRx$ but  $F[x]$ is neither contained in the center nor
contains a nonzero ideal of $R$. Therefore, by     \cite[Theorem]{H},   the ring  $R$ does not contain nontrivial idempotents. Thus $R$ is an abelian   p.q.-Baer ring. We have  already  seen in Example \ref{example Cohn} that $R$ is neither left nor right fusible.
\end{ex}

Example \ref{example Cohn}  also suggests that although   abelian p.q.-Baer rings need not be  fusible but maybe they have to be regular fusible. We are unable to answer the following question:
 \begin{qu}
  Is every abelian left  p.q.-Baer ring regular left fusible?
 \end{qu}

Now we will move to some basic properties of regular left fusible rings. Recall that a left ideal $I$ of a ring $R$ is essential (denoted by $I<_e  {_R}R$) if $A\cap I \neq 0$ for every nonzero left ideal $A$ of $R$.

\begin{lem}\label{lem ann a-b}
 Let $a,b\in R$ be such that $\la{a}<_e {_R}R$ and $\la{b}\ne 0$. Then $\la{a-b}\ne 0$. In particular, every element  element $a$ such that  $\la{a}<_e {_R}R$ can not be left fusible.
\end{lem}
\begin{proof}
 The imposed assumptions imply that  $\la{a}\cap\la{b}\ne 0$, so the thesis is clear.
\end{proof}

The left singular ideal of a ring $R$ is $ \lS{R}=\{x\in R\mid \la{x}<_e {_R}R \}$. It is well known that $\lS{R}$ is a two-sided ideal of $R$. A ring $R$ is called left nonsingular if   $\lS{R}=0$.

The following  lemma  generalizes  \cite[Proposition 2.11]{GM}.
\begin{lem}\label{lem singular}
 Every regular left fusible ring $R$ is left nonsingular.
\end{lem}
\begin{proof}
Assume $0\ne a\in \lS{R}$.  Since $R$ is regular left fusible, we can pick a regular element  $s\in R$  such that $sa$ is left fusible. Since $\lS{R}$ is a two-sided ideal, $sa\in \lS{R}$. Thus,  by    Lemma \ref{lem ann a-b}, $sa$ is not left fusible,   a contradiction.
\end{proof}

Let $S$ denote the left Ore set consisting of regular element of the ring $R$ and let $S^{-1}R$ be the left Ore localization of $R$. Writing $Q_l(R)$ we will mean that $R$ possesses    a classical left quotient ring which is equal to $Q_l(R)$.
  \begin{lem} \label{zero divisors}  Let $s\in S$ and $a \in R$. Then
      $\La{S^{-1}R}{s^{-1}a  } \neq 0$   if and only if $\la{a}\neq 0$.

      \end{lem}
   \begin{proof}
      Suppose $bs^{-1}a=0$, for some nonzero $b\in S^{-1}R$. Eventually multiplying on the left by a suitable element from $S$,  we may assume $0\ne b\in R$. Since $S$ is the left Ore set, there exists $t\in S$ and $0\ne c\in R$ such that $bs^{-1}=t^{-1}c$. Then $0\ne c\in \la{a}$.

    Suppose $ca=0$, for some $0\ne c \in R$. Then $(cs) s^{-1}a=0$, so $s^{-1}a$ is a left zero divisor in  $S^{-1}R$.
       \end{proof}

     \begin{prop}\label{main prop} Let $R$ be a ring.
     \begin{enumerate}
       \item If $R$ is left fusible, then so is $S^{-1}R$.
       \item If $R$ is regular left fusible, then $Q_l(R)$ is left fusible.
       \item  If $S^{-1}R$ is regular left fusible, then $R$ is regular left fusible.
       \item Suppose every regular element  of $R$ is invertible. Then   $R$ is regular left fusible iff   $ R$ is left fusible.
     \end{enumerate}

    \end{prop}
  \begin{proof}
(1) Let $s^{-1}r\in S^{-1}R$. As $R$ is left fusible, $r$ has a left fusible decomposition, say  $r=a+w$. Lemma \ref{zero divisors} shows that $s^{-1}r=s^{-1}a+s^{-1}w$ is a left fusible decomposition of $s^{-1}r$.

(2) Suppose $R$ is regular left fusible. Let $0\ne s^{-1}r\in Q_l(R)$. Then there exists a regular element $t\in R$ such that $tr$ has a fusible decomposition $tr=a+w$. By Lemma  \ref{zero divisors}, $s^{-1}r=(ts)^{-1}a+(ts)^{-1}w$ is a fusible decomposition of $s^{-1}r$ in $Q_l(R)$.

(3) Let $r\in R$. Since $S^{-1}R$ is regular left fusible, there exists a regular element $s_1^{-1}t_1\in S^{-1}R$ such that the element $z={s_1}^{-1}t_1r$ has a left fusible decomposition, say ${s_1}^{-1}t_1r={s_2}^{-1}a_2+ {s_3}^{-1}w_3$.  Taking the left common denominator,   we may  write $z=s^{-1}tr=s^{-1}a+ s^{-1}w$ for suitable $s\in S,\;  t,a,w\in R$. By Lemma  \ref{zero divisors}, $t $ is  regular, $w$ is  left regular and $a $ is  a left zero divisor of $R$. This implies that the element $tr$ is  left fusible in $R$  and hence $R$ is regular left fusible.

(4) By assumption, $R=Q_l(R)$.  Let $r\in R$. Suppose $R$ is regular left fusible. There exists a regular element $t\in R$ such that $tr$ has a fusible decomposition $tr=a+w$.    Then, by Lemma \ref{zero divisors},  $   r= t^{-1}a+t^{-1}w $ is a fusible decomposition of $r$. This yields  the proof of (4).
  \end{proof}
The above proposition gives immediately the following:

\begin{cor} \label{equivalence reg. l.fus} Suppose $R$ has the left quotient ring $Q_l(R)$. The following conditions are equivalent:\begin{enumerate}
\item $R$ is regular left fusible.
\item  $Q_l(R)$ is left fusible.
\item   $Q_l(R)$ is regular left fusible.

                                          \end{enumerate}

\end{cor}

In \cite[Propsition 3.14]{GM},  the authors observed that the statement $(1)$ of Proposition \ref{main prop}  holds when $R$ is a commutative ring and asked whether the reverse implication holds. In the context of Corollary \ref{equivalence reg. l.fus}, their question can be restated as:

\begin{qu}\label{question 1}
Let  $R$ be a regular left fusible ring having  the classical left quotient ring.  Is $R$   left fusible?
\end{qu}

 Notice that, although the ring  $R$  in Example  \ref{example Cohn} is not left fusible, the ring $R$ is regular left fusible. Thus  the above question has     a negative answer without the assumption that $R$ possesses the classical  left quotient ring.

Using the same arguments as in \cite[Lemma 2.19.]{GM} and making use of the fact that in a commutative ring $R$ an element $r$ is nilpotent if and only if the ideal  $rR$ is nilpotent we   get the following modification of \cite[Lemma 2.19.]{GM}.
\begin{lem}\label{comm reguced}
 Every commutative  regular fusible ring  is reduced.
\end{lem}

In \cite[Propsition 3.7]{GM},  the authors showed that a commutative ring  with only finitely many minimal prime ideals is fusible  if and only
if it is reduced. Using this characterization we obtain a positive, partial answer to Question \ref{question 1}:

\begin{prop}\label{min-prime} Let   $R$ be a commutative ring with only finitely many minimal prime ideals. The following conditions are equivalent:
 \begin{enumerate}
                                            \item  $R$ is   fusible.
                                            \item  $R$ is regular   fusible.
                                            \item  $R$ is reduced.
                                          \end{enumerate}
\end{prop}
\begin{proof} The implication $(1)\Rightarrow (2)$ is a tautology.
The implications $(2)\Rightarrow (3)$  and $(3)\Rightarrow (1)$ are given by Lemma \ref{comm reguced} and \cite[Propsition 3.7]{GM}, respectively.
\end{proof}
  \cite[Example 3.13]{GM} offers a  commutative ring   which is reduced but   is not fusible, i.e.  the equivalence $(1)\iff (3)$ does not hold for arbitrary commutative rings. The ring from this example is also not regular fusible. We do not know of any example of a commutative ring which is regular fusible but not fusible.

A ring $R$ is called unit-regular if for each $a \in R$ there exists a unit $u \in U(R)$ such that $aua = a$. In \cite[Proposition 5.3]{GM}, the following proposition was proved. We offer a short direct argument.
\begin{prop}\label{unit regular}
 Every unit-regular ring is left fusible.
\end{prop}
\begin{proof}
 Let $R$ be a unit-regular ring and $r\in R$, i.e., $0\ne r=rur$ where $u\in R$.  Then $e=ur$ is an idempotent and $e=(-1+e) + (1-2e)$ is its fusible decomposition. This shows that $R$ is regular left fusible. Now, since every regular element in $R$ is invertible,  the thesis is a direct consequence of Proposition \ref{main prop}(3).
\end{proof}

 Proposition \ref{main prop} and Corollary \ref{equivalence reg. l.fus}  enable us to give the following  new characterization of semiprime left Goldie rings:

\begin{thm}\label{Goldie}
For a semiprime ring $R$, the following conditions are equivalent:
\begin{enumerate}
  \item $R$ is left Goldie.
  \item $R$ is regular left fusible and has finite left Goldie dimension
\end{enumerate}
\end{thm}
\begin{proof}
 Suppose $R$ is left Goldie. Then, by Goldie's Theorem, $Q_l(R)$ is a semisimple artinian ring. In particular it has finite left Goldie dimension and it is a unit-regular ring. Thus, by Proposition \ref{main prop}, $Q_l(R)$ is  left fusible.    Now Corollary \ref{equivalence reg. l.fus}  yields that $R$ is regular left fusible.  This completes the proof of $(1)\Rightarrow (2)$.

 Let $R$ be as in $(2)$. Then, by Lemma \ref{lem singular}, $R$ is left nonsingular. Now the implication $(2)\Rightarrow (1)$ is a consequence of Goldie's Theorem.
\end{proof}
Recall  that  a ring  $R$ a right complemented, if for each $a\in R$, there is a
$b\in R$ such that $ab = 0$ and $a + b$ is regular (see \cite{GM}). By  \cite[Proposition 3.4]{GM}, every right complemented ring is left fusible.
  \cite[Question (3)]{GM}  asks whether  a right complemented ring have always possesses  a right classical quotient ring?
 Since any domain is   complemented on both sides and there are domains having no a classical right quotient ring, the answer to this question is no in general. However  the above theorem gives immediately the following:

 \begin{cor}
  Every right (left) complemented ring of finite left Goldie dimension has a classical left quotient ring.
 \end{cor}
\begin{proof}
As observed in \cite[Proposition 3.4]{GM},  every right (left) complemented ring is reduced, so  it is semiprime. Now the thesis is a consequence of the left version of Theorem \ref{Goldie}.
\end{proof}

  One can easily adopt the proof of    \cite[Theorem 2.18]{GM}
to get the following:

\begin{prop}\label{prop. matrices}
  Suppose that for any finite set $r_1,\ldots, r_n\in R\setminus \{0\}$ there exists a regular element $s\in R$ such that the elements $sr_1,\ldots, sr_n$ are left fusible. Then the matrix ring $ M_n(R)$ is regular left fusible.
  \end{prop}

With the help of the above proposition we get:
\begin{thm}\label{rlf matrices}
  Suppose $R$ is a regular left fusible ring having the classical left quotient ring. Then the matrix ring $ M_n(R)$ is regular left fusible.
\end{thm}
\begin{proof}
 By Corollary \ref{equivalence reg. l.fus}, $Q_l(R)$ is left fusible and   Proposition \ref{prop. matrices} shows that $M_n(Q_l(R))$ is regular left fusible.
 Let us notice that $M_n(Q_l(R))=S^{-1}M_n(R)$, where $S$ denotes the set of all diagonal matrices $diag(s,\ldots,s)$, where $s$ ranges over all regular elements of $R$. Now the thesis is a direct consequence of Proposition \ref{main prop}(3).
\end{proof}

Let $R$ be the ring from Example \ref{example Cohn}. Then $y\in R$ is regular and,   for any $0\ne r\in R$, $yr$ is left regular, so $R$ satisfies the assumptions of Proposition \ref{prop. matrices}. Therefore we have:
\begin{ex}
  Let $R$ be the ring from Example \ref{example Cohn}. Then the matrix ring $M_n(R)$ is regular fusible, for any $n\geq 1$.
\end{ex}

We were unable to answer the following:
\begin{qu}
 Is the matrix ring $M_n(R)$ over a regular left fusible ring $R$ itself  regular left fusible?
\end{qu}

We close the paper with an observation that regular left fusible property behaves well under polynomial extensions.  Essentially the same proof  as in \cite[Proposition 2.9]{GM}
gives the following
\begin{prop}\label{poly}
 If $R$ is regular left fusible, then so are the rings $R[x;\sigma], R[x,x^{-1};\sigma]$ and $R[[x;\sigma]] $, where $\sigma$ is  an automorphism of $R$.
\end{prop}

\bigskip
\end{document}